\documentclass[11pt]{article}
\usepackage[width=16cm,height=21cm]{geometry}
\usepackage{amsmath,amsthm}

 \newtheorem{thm}{Theorem}[section]
 \newtheorem{cor}[thm]{Corollary}
 \newtheorem{lem}[thm]{Lemma}
 \newtheorem{prop}[thm]{Proposition}
 \theoremstyle{definition}
 \newtheorem{defn}[thm]{Definition}
 \theoremstyle{remark}
 \newtheorem{rem}[thm]{Remark}
 \newtheorem{ex}[thm]{Example}
 \numberwithin{equation}{section}

\usepackage[colorlinks=true,linkcolor=blue,citecolor=blue]{hyperref}
\newcommand{\myurl}[1]{\href{#1}{#1}}

\usepackage{amssymb}

\usepackage{cite}

\newcommand{\bC}{\mathbb{C}}
\newcommand{\bD}{\mathbb{D}}
\newcommand{\bR}{\mathbb{R}}
\newcommand{\bN}{\mathbb{N}}
\newcommand{\bNz}{\mathbb{N}_0}
\newcommand{\bT}{\mathbb{T}}
\newcommand{\bZ}{\mathbb{Z}}

\newcommand{\cB}{\mathcal{B}}
\newcommand{\cC}{\mathcal{C}}

\newcommand{\al}{\alpha}
\newcommand{\be}{\beta}

\newcommand{\ga}{\gamma}

\newcommand{\ka}{\varkappa}

\newcommand{\la}{\lambda}

\newcommand{\Tht}{\Theta}
\newcommand{\tht}{\vartheta}
\renewcommand{\phi}{\varphi}

\newcommand{\enumber}{\operatorname{e}}
\newcommand{\iu}{\operatorname{i}}

\newcommand{\dif}{\mathrm{d}}

\newcommand{\Ber}{\operatorname{Ber}}

\usepackage{mathtools}
\newcommand{\eqdef}{\coloneqq}

\begin{document}

\title{Toeplitz operators in Bergman space\\
induced by radial measures}

\author{Egor A. Maximenko, Carlos G. Pacheco}

\maketitle

\begin{center}
We dedicate this article in the memory to Prof. Nikolai Vasilevski,\\
good colleague and friend.
\end{center}

\begin{abstract}
We study radial Carleson--Bergman measures
on the unit disk
and the corresponding
Toeplitz operators acting in the Bergman space.
First, we show that such Toeplitz operators are diagonal in the canonical basis, and we compute their eigenvalue sequences and Berezin transforms
in terms of the radial component of the measure.
Next, considering the average values of radial measures near the boundary, we give a simple characterization of radial Carleson--Bergman measures.
Finally, we prove that the eigenvalue sequences of such Toeplitz operators are Lipschitz continuous with respect to the logarithmic distance on natural numbers.
As a consequence, we describe the commutative C*-algebra generated by Toeplitz operators induced by radial Carleson--Bergman measures.



\end{abstract}

\medskip
\subsection*{Autors' data}

Egor A. Maximenko
(\myurl{https://orcid.org/0000-0002-1497-4338}),\\
Escuela Superior de F\'isica y Matemáticas,
Instituto Polit\'ecnico Nacional,\\
Postal Code 07730,
Ciudad de M\'{e}xico, Mexico\\
egormaximenko@gmail.com, emaximenko@ipn.mx

\medskip\noindent
Carlos G. Pacheco
(\myurl{https://orcid.org/0000-0002-5528-2653}),\\
Departamento de Matem\'{a}ticas,
Cinvestav,\\
Postal Code 07360,
Ciudad de M\'{e}xico, Mexico\\
carlos.pacheco@cinvestav.mx

\section{Introduction}

Let $\bD$ be the unit disk in the complex plane
with the usual Lebesgue plane measure $\lambda_\bC$.
Let $A^2(\bD)$ be the Bergman subspace of $L_2(\bD)$,
consisting of all analytic and square-integrable functions.
If $a\in L_\infty(\bD)$,
then the Toeplitz operator $T_{a}$ acting on $A^{2}(\bD)$,
with symbol $a$,
is the composition of the multiplication by $a$ and the projecting operaror to the subspace $A^{2}(\bD)$:
\[
T_{a}f= P(af).
\]
It is known that this operator can be expressed in terms of the reproducing kernel of the Bergman space:
\begin{equation}\label{EqTa}
T_{a}f(z)
=
\frac{1}{\pi}
\int_{\bD}
\frac{a(w)f(w)}{(1-z \overline{w})^{2}}\,
\dif{}\lambda_\bC(w).
\end{equation}
Korenblum and Zhu~\cite{KorenblumZhu1995} noticed
that if $a$ is an essentially bounded function on $[0,1)$,
naturally extended onto $\bD$ by $a(z)\eqdef a(|z|)$,
then $T_a$ is diagonal
with respect to the canonical basis of $A^2(\bD)$,
and the corresponding eigenvalues are
\begin{equation}
\label{eq:eigenvalues_radial_Toeplitz_bounded_symbol}
\gamma_a(n)
=2(n+1) \int_{[0,1)} a(r) r^{2n+1}\,\dif{}r
\quad(n\in\bNz\eqdef\{0,1,2,\ldots\}).
\end{equation}
The main result of~\cite{KorenblumZhu1995} was a compactness criterion for such operators $T_a$.

Grudsky and Vasilevski~\cite{GrudskyVasilevski2001}
studied Toeplitz operators
with unbounded radial defining symbols $a$.
In particular,
for the case of positive functions $a$,
they found a characterization of
their boundedness and compactness.

In this paper, we extend some ideas from~\cite{GrudskyVasilevski2001}
to Toeplitz operators induced by radial measures.

If $\mu$ is a finite Borel measure on $\bD$,
then $T_\mu$ is defined on a dense subset of $A^2(\bD)$ by
\begin{equation}
\label{EqTmu}
T_{\mu}f(z)
\eqdef
\frac{1}{\pi}
\int_{\bD}
\frac{f(w)}{(1-z \overline{w})^{2}}\,\dif{}\mu(w).
\end{equation}
Following Zhu~\cite[Section~7.2]{Zhu2007},
we say that a finite Borel measure $\mu$ on $\bD$
is a \emph{Carleson type measure for $A^2(\bD)$} or,
shortly, \emph{Carleson--Bergman measure},
if there exists a constant $C\ge0$ such that
for every $f$ in $A^2(\bD)$,
\[
\int_{\bD}|f(z)|^{2}\,\dif{}\mu(z)
\leq
C \int_{\bD}|f(z)|^{2}\,\dif{}\lambda_\bC(z).
\]
The reader may consult~\cite[Section~7.2]{Zhu2007}
for other equivalent descriptions.
It is known~\cite[Theorem~7.5]{Zhu2007}
that the operator $T_\mu$ is bounded
if and only if $\mu$ is a Carleson--Bergman measure.
Moreover,
if $\mu$ is a Carleson--Bergman measure,
then the sesquilinear form associated to $T_\mu$
can be represented as the following integral
(see \cite[note after Theorem~7.5]{Zhu2007}):
\begin{equation}
\label{eq:Toeplitz_measure_sesquilinear_form}
\langle T_\mu f,g\rangle
=
\int_\bD f\,\overline{g}\,\dif\mu
\qquad(f,g\in A^2(\bD)).
\end{equation}

We study the situation
when $\mu$ is a \emph{radial measure},
which means that $\mu$ is invariant under all rotations
associated to the unit circle
$\bT\eqdef\{z\in\bC\colon |z|=1\}$.
In Subsection~\ref{subsec:RadM}
we show that every finite radial Borel measure on $\bD$,
after passing from $\bD$ to $[0,1)\times\bT$,
can be decomposed into the product of two measures:
one finite Borel measure $\eta$ on $[0,1)$
and the Lebesgue (or Haar) measure on $\bT$.
Furthermore, in Subsections~\ref{subsec:canonical_basis}
and~\ref{subsec:RadOper}
we recall the canonical isomorphism between $A^2(\bD)$,
the concept of radial operators
and their diagonalization.

In Section~\ref{sec:diagonalization_Toeplitz_radial_measure},
supposing that $\mu$ is a finite radial Carleson--Bergman measure on $\bD$,
we show that the corresponding Toeplitz operator $T_\mu$ is diagonal with respect to the canonical basis,
and compute its eigenvalue sequence $\gamma_\eta$
in terms of the measure $\eta$,
which is the ``radial component'' of $\mu$.
When $\dif\eta(r)=a(r)\,r\,\dif{}r$,
where $a\in L_\infty([0,1))$ and $a\ge0$,
our formula \eqref{EqTmu} takes the form \eqref{EqTa}.
In Section~\ref{sec:Berezin}
we express the Berezin transform of $\mu$ in terms of $\eta$
and also in terms of $\ga_\eta$.

In Section~\ref{sec:AF}
we introduce a function $\ka_\eta\colon[0,1)\to[0,+\infty)$
which gives average values of $\eta$ at the boundary.
Using the integration by parts,
we express $\ga_\eta$ and the Berezin transform of $\eta$ in terms of $\ka_\eta$.

In Section~\ref{sec:radial_measure_criterion_Carleson}
we show that for any finite radial Borel measure $\mu$,
the condition ``$\mu$ is Carleson--Bergman''
is equivalent to the boundedness of $\ka_\eta$
(and also to the boundedness of $\ga_\eta$).

In Section~\ref{sec:Lip}
we show that if $\ka_\eta$ is bounded,
then $\ga_\eta$ is Lipschitz continuous with respect to the logarithmic distance $d_{\log}$ on $\bNz$.

Finally, in Section~\ref{sec:class_of_radial_complex_measures},
we trivially extend some of the previous results to the case of complex radial measures
that are complex linear combinations of positive radial Carleson--Bergman measures.
In particular, we conclude that the corresponding eigensequences $\gamma_\eta$ generate the same C*-algebra as the eigensequences $\gamma_a$ with $a$ in $L_\infty([0,1))$.
The latter C*-algebra was described in
\cite{Suarez2008,
GrudskyMaximenkoVasilevski2013,
BauerHerreraYanezVasilevski2014,
HerreraMaximenkoVasilevski2015}.

We hope that this paper will
inspire other investigations
on special classes of Carleson type measures on Bergman, Fock and other similar spaces.
We recall that Vasilevski and his colleagues found various classes of bounded functions on $\bD$
and on the upper half-plane
that induce commutative C*-algebras of Toeplitz operators;
see, e.g.,
\cite{Vasilevski2003,GrudskyKarapetyantsVasilevski2003,GrudskyQuirogaVasilevski2006,Vasilevski2008}.
A natural question is to find classes of measures that define commutative classes of operators.
Moreover, one can consider Toeplitz operators induced by more abstract defining symbols, such as distributions;
see, e.g.,~\cite{RozenblumVasilevski2016,
Suarez2015,
EsmeralRozenblumVasilevski2019}.

\section{Radial measures and radial operators}
\label{sec:radial_measures_and_radial_operators}

\subsection{Radial measures}
\label{subsec:RadM}

Radial measures have been used by various authors,
especially in the context of weight functions for function spaces
(see, e.g., \cite{DalMasoMosco1985,HachadiYoussfi2019}).
In this subsection, we give a formal and complete treatment of this concept.

The intuitive idea of a radial measure comes from a radial function, which is that the values of it do not change along the circumference.
One can see that this idea is also related to having measure invariance when rotating a Borel set; and this helps to propose a proper definition of radial measure, which is given in Definition~\ref{DefRM}.

We start by giving some basic ingredients.
Given a topological space $T$,
we denote by $\cB_T$ the set of all Borel subsets of $T$.
Here we consider finite Borel measures on $\cB_T$, where $T$ can be $\bD$, $[0,1)$ or $\bT\eqdef\{z\colon\ |z|=1\}$.
Since these spaces are Polish, the measures are regular, see, e.g., \cite[Theorem 17.7]{Aliprantis}), which is relevant for us because we also consider Haar measures.

\begin{defn}[radial measure]\label{DefRM}
Let $\mu$ be a Borel measure on $\bD$.
We say that $\mu$ is \emph{radial},
if for every $\tau$ in $\bT$
and every $Z$ in $\cB_\bD$,
\[
\mu(\tau Z) = \mu(Z).
\]
\end{defn}

Now, we turn to another intuitive idea of how a radial measure can be decompose into two parts,
one is a measure $\eta$ on the unit interval, which is precisely the radial part,
and the other one is a Haar measure $\mu_{\bT}$ on the circumference.
Let us see rigorously how this equivalence holds.

We denote by $\Tht$ the radial change of variables:
\[
\Tht \colon [0,1)\times\bT \to \bD,
\qquad \Tht(\tau,r) \eqdef r\tau.
\]

Let $\la_\bT\colon\cB_{\bT}\to[0,1]$ be the Haar measure on $\bT$
such that $\la_\bT(\bT)=2\pi$.
Equivalently, given $X$ in $\cB_\bT$,
\[
\la_\bT(X)
\eqdef
\la\left(\bigl\{\tht\in[0,2\pi)\colon\ \enumber^{\iu\tht}\in X\bigr\}\right).
\]
In this paper, $\la$ is the Lebesgue measure on $\bR$.

\begin{defn}[radial extension of a measure]\label{DefRadialExt}
Let $\eta\colon\cB_{[0,1)}\to[0,+\infty)$ be a Borel measure,
and consider its extension over $[0,1]$ by taking $\eta(\{1\})=0$.
Let $\eta\times\la_{\bT}$ represents the product measure.
We define a Borel measure $\nu_\eta\colon\cB_{\bD}\to[0,+\infty)$ as the following pushforward measure: 
\begin{equation}
\label{eq:radial_measure_extension}
\nu_\eta(Z)
\eqdef
(\eta\times\la_{\bT}) (\Tht^{-1}(Z))
\qquad(Z\in \cB_{\bD}).
\end{equation}
In this case we say that $\nu_\eta$ is \emph{the radial extension} of $\eta$.
\end{defn}

\begin{rem}
\label{rem:regular_measure}
The reason of taking $\eta(\{1\})=0$ in previous definition is to work in a compact space, which helps guarantee that the measure $\nu_{\eta}$ is regular, as can be seen in Aliprantis~\cite[Theorem 12.10]{Aliprantis}.
\end{rem}

Now we can link Definitions~\ref{DefRM} and \ref{DefRadialExt}.

\begin{thm}[decomposition of a radial measure]
\label{ThmERM}
Let $\mu\colon\cB_\bD\to[0,+\infty)$ be a finite Borel measure.
Then the following conditions are equivalent:
\begin{itemize}
\item[(a)] $\mu$ is radial;
\item[(b)] there exists a finite Borel measure $\eta\colon\cB_{[0,1)}\to[0,+\infty)$, such that 
$\mu$ is its radial extension, i.e., $\mu=\nu_\eta$.
\end{itemize}
\end{thm}

\begin{proof}
(b)$\Rightarrow$(a).
Suppose that $\eta$ is a Borel measure on $[0,1)$
and $\mu=\nu_\eta$.
Given $Z$ in $\cB_\bD$ and $\tau$ in $\bT$, we have
\begin{align*}
\Tht^{-1}(\tau Z)
&=\bigl\{(r,t)\in[0,1)\times\bT\colon\ rt\in\tau Z\bigr\}
\\
&=\bigl\{(r,t)\in[0,1)\times\bT\colon\ rt\tau^{-1}\in Z\bigr\}
\\
&=\bigl\{(r,t\tau)\colon\ (r,t)\in\Tht^{-1}(Z)\bigr\}.
\end{align*}
As mentioned in Remark~\ref{rem:regular_measure},
we know that $\nu_{\eta}$ is regular,
thus it is enough to study sets $Z=\Tht(X\times Y)$, where $X$ and $Y$ are Borel set in $[0,1]$ and $\bT$, respectively.
Since $\la_\bT$ is invariant under rotations,
\[
\nu_\eta(\tau Z)
=(\eta\times\la_{\bT}) (\Tht^{-1}(Z))
=\eta (X) \la_{\bT} (\tau Y)
=\eta (X) \la_{\bT} (Y)
=\nu_\eta(Z).
\]
(a)$\Rightarrow$(b).
Suppose that $\mu$ is a radial finite Borel measure on $\bD$.
For every $X$ in $\cB_{[0,1)}$,
we define $\phi_X\colon\cB_\bT\to[0,+\infty)$
by
\[
\phi_X(Y)
\eqdef \mu(\Tht(X\times Y))
\qquad(X\in\cB_\bT).
\]
Since $\Tht$ is continuous,
$\phi_X$ is also a regular Borel measure,
see \cite[Theorem 12.10]{Aliprantis}.
If $\tau\in\bT$,
then for every $Y$ in $\cB_\bT$,
\[
\Tht(X \times (\tau Y))
= \tau \Tht(X\times Y).
\]
Using the assumption that $\mu$ is radial, we get
\[
\phi_X(\tau Y)
= \mu(\Tht(X \times (\tau Y)))
= \mu(\Tht(X\times Y))
= \phi_X(Y).
\]
Thus,
$\phi_X$ is invariant under the group operation in $\bT$;
that is,
$\phi_X$ is a Haar measure on $\bT$.
Since the Haar measure is unique up to multiplication by strictly positive constants,
$\phi_X$ is a multiple of the Haar measure $\mu_\bT$.
It means that there exists a number
which we denote by $\eta(X)$
such that for every $Y$ in $\cB_\bT$,
\[
\phi_X(X) = \eta(X) \la_\bT(Y).
\]
In particular,
\[
\eta(X)
=
\mu(\Tht(X \times \bT)).
\]
Now, one can check that $\eta$
is a Borel measure on $[0,1)$.

Let us show that $\nu_\eta=\mu$.
For every $Z$ in $\cB_{\bD}$
of the form $Z=\Tht(X\times Y)$,
where $X\in\cB_{[0,1)}$ and $Y\in\cB_\bT$ and ,
\[
\nu_\eta(Z)
= \eta(X)\,\la_\bT(Y)
= \phi_X(Y)
= \mu(\Tht(X\times Y))
= \mu(Z).
\]
Since $\mu$ and $\nu_\eta$ are regular Borel measures,
we conclude that $\nu_\eta = \mu$.
\end{proof}

\begin{rem}
\label{rem:integral_wrt_radial_measure}
If $\eta$ is a positive Borel measure on $[0,1)$
and $\mu=\nu_\eta$,
then the integrals with respect to $\mu$
can be written in the following form:
\begin{equation}
\label{eq:integral_with_respect_to_radial_measure}
\int_{\bD} f\,\dif\mu
=
\int_{[0,1)} \int_{\bT}
f(r\tau)\,\dif\nu(r)\,\dif\mu_\bT(\tau).
\end{equation}
This formula can be easily checked for characteristic functions
and then extended to positive measurable functions
and even to functions of the class $L_1(\bD,\mu)$.
\end{rem}

\begin{ex}
\label{ex:eta0}
Let $\eta_0$ be the measure defined on $[0,1)$ by
$\dif\eta_0(r)\eqdef r\,\dif{}r$.
Then $\nu_{\eta_0}$
is the usual Lebesgue measure
$\la_\bC$ on $\bD$.
\end{ex}

\subsection{Canonical basis of Bergman space}
\label{subsec:canonical_basis}

For each $k$ in $\bNz$,
we define $b_k\colon\bD\to\bC$ by
\begin{equation}
\label{eq:canonical_basis}
b_k(z)\eqdef \sqrt{\frac{k+1}{\pi}}\,z^k.
\end{equation}
Furthermore, we denote by $b$ the sequence
$(b_k)_{k\in\bNz}
=(b_0,b_1,b_2,\ldots)$
of these functions.
It is well known (see, e.g., \cite[proof of Corollary~4.20]{Zhu2007})
that $b$ is an orthonormal basis of $A^2(\bD)$.
We call it the \emph{canonical basis} of $A^2(\bD)$.

For each $k$ in $\bNz$,
let $e_k=(\delta_{k,n})_{n\in\bNz}$,
where $\delta$ is the Kronecker delta.
It is well known that $(e_k)_{k\in\bNz}$
is an orthonormal basis of $\ell^2(\bNz)$.

Following ideas by Vasilevski~\cite{Vasilevski2003},
we briefly explain a natural isomorphism between $A^2(\bD)$ and $l^2(\bNz)$ in terms of the following three operators.

$\bullet$.
The operator $U_{0}$ maps $L_{2}(\bD)$
onto its tensor product representation
$L_2([0,1)\times\bT, \eta_0\times\lambda_{\bT})
=L_{2}([0,1), \eta_{0}(r))\otimes L_{2}(\bT)$,
using the change of variables $\Tht$.
Hence, for any $f$ in $L_{2}(\bD)$,
\begin{equation*}
(U_{0}f)(r,\tau) \eqdef f(r\tau)
\qquad(r\in[0,1),\ \tau\in\bT).
\end{equation*}

$\bullet$.
The operator
$U_1\colon L_2([0,1)\times\bT, \eta_0\times\lambda_{\bT})\to \widetilde{H}$
is the inverse of the Fourier transform
acting in the second component:
\begin{equation*}
(U_1 g)(r,k)
\eqdef
\frac{1}{\sqrt{2\pi}}
\int_{\bT} g(r,\tau)\tau^{-k}
\dif\lambda_{\bT}(\tau).
\end{equation*}
Here $\widetilde{H}\eqdef L_2([0,1),\eta_0)\otimes l_2(\bZ)$.

$\bullet$
The operator $U_{2}\colon \widetilde{H}\to \ell_2(\bZ)$
acts by
\begin{equation*}
(U_{2}h)(k)
\eqdef \sqrt{2(k+1)}\int_{[0,1)} h(r,k)r^{k+1}\,\dif\lambda(r).
\end{equation*}
It is easy to see that if $f\in A^2(\bD)$,
then $U_2 U_1 U_0 \in \ell_2(\bNz)$.

We denote by $R$ the composition $U_2 U_1 U_0$ with domain restricted to $A_2(\bD)$ and codomain restricted to $\ell_2(\bNz)$.

\begin{prop}
\label{PropRf}
1. For every $f$ in $L_{2}(\bD)$ and $k$ in $\bNz$,
\[
(Rf)(k)=\langle f,b_{k} \rangle.
\]
2. For every $k$ in $\bNz$,
\[
Rb_{k}=e_{k} \quad \text{and}\quad R^\ast e_{k}=b_{k}.
\]
3. $R$ is an isometric isomorphism between $A^2(\bD)$ and $\ell_2(\bNz)$.
\end{prop}

\begin{proof}
Combining $U_{0}$, $U_{1}$ and $U_{2}$, we obtain
\begin{align*}
(Rf)(k)
&=
\frac{\sqrt{2(k+1)}}{\sqrt{2\pi}}
\int_{[0,1)} \int_{\bT}
f(r\tau)\tau^{-k}r^{k+1}\,\dif\mu_{\bT}(\tau)\,\dif{}r
\\
&=
\int_{[0,1)}\int_{\bT}
f(r\tau)\frac{\sqrt{k+1}\,\tau^{-k}r^{k}}{\sqrt{\pi}}\dif\mu_{\bT}(\tau)\,r\,\dif{}r
\\
&=
\int_{\bD} f(z)\overline{b_{k}(z)}\,\dif\la_\bC(z).
\qedhere
\end{align*}
The second and third point are consequences of the first one.
\end{proof}

\subsection{Radial operators}
\label{subsec:RadOper}

In this subsection,
we recall the concept of bounded radial operators
acting in $A^2(\bD)$.
Zorboska~\cite{Zorboska2002} characterized bounded radial operators in terms of the ``radialization transform''.
It is known~\cite{GrudskyMaximenkoVasilevski2013, Quiroga2016}
that a bounded operator acting in $A^2(\bD)$ is radial
if and only if it is diagonal with respect to the canonical basis defined by~\eqref{eq:canonical_basis}.
For completeness,
we give a short elementary proof of this fact.
We denote by $\cB(A^2(\bD))$
the algebra of all bounded operators on $A^2(\bD)$.

\medskip
For every $\tau$ in $\bT$,
we denote by $\rho_{\tau}$ the rotation operator
acting in $A^2(\bD)$ by the following rule:
\begin{equation}\label{Eqrhotau}
(\rho_{\tau} f)(z)
\eqdef
f(\tau^{-1}\,z).
\end{equation}
It is worth to notice the following.
One can see that $(\rho_{.},A^2(\bD))$ is a unitary representation of the group $\bT$. 
Quiroga-Barranco~\cite{Quiroga2016}
applied some tools from representation theory to study radial and separately radial operators.

\begin{defn}
Define $\cC(\rho)$, which is called the centralizer of $\rho$, as the collection of operators that commute with every rotation operator:
\[
\cC(\rho)
\eqdef
\bigl\{
S\in\cB(A^2(\bD))\colon\quad
\forall \tau\in\bT\quad \rho_{\tau}S=S\rho_{\tau}
\bigr\}.
\]
The elements of $\cC(\rho)$ are called \emph{radial operators}.
\end{defn}

\begin{defn}\label{DefiDiago}
Let $S\in\cB(A^2(\bD))$.
The operator $S$ is called \emph{diagonal} with respect to $b$,
if there exists $\al$ in $\ell_\infty(\bNz)$ such that
\begin{equation}
\label{eq:Sb_eq_al_b}
\forall j\in\bNz\qquad Sb_j = \al_j b_j.
\end{equation}
\end{defn}

Recall that given a sequence $\al$ of the class $l_\infty(\bNz)$,
the multiplication operator $M_\al$ on $l_{2}(\bNz)$
is defined by
\[
M_{\al}(c) \eqdef (\al_k c_k)_{k\in\bNz}
\qquad(c\in l_2(\bNz)).
\]
Taking into account Proposition~\ref{PropRf},
\eqref{eq:Sb_eq_al_b} is equivalent to
$R S R^\ast = M_\al$.

\begin{lem}
\label{lem:rho_basic}
For every $j$ in $\bNz$ and every $\tau$ in $\bT$,
$\rho_{\tau}b_j = \tau^{-j}b_j$.
\end{lem}

\begin{proof}
Indeed,
$(\rho_{\tau}b_j)(z)
=b_j(\tau^{-1}z)
=\sqrt{\frac{j+1}{\pi}}\,\tau^{-j}z^j
=\tau^{-j}b_j(z)$.
\end{proof}

\begin{prop}\label{PropEquiv}
Let $S\in\cB(A^2(\bD))$.
Then $S$ is radial if and only if
$S$ is diagonal with respect to $b$.
\end{prop}

\begin{proof}
($\Rightarrow$)
Let $S\in\cC(\rho)$.
For every $j$ in $\bNz$, we put
$\al_j\eqdef\langle Sb_j,b_j\rangle$.
Then $|\al_j|\le\|S\|$ for every $j$.
Therefore, $\al\in l^\infty(\bN)$.
For every $k$ in $\bN$ with $k\ne j$,
we will show that $\langle Sb_j,b_k\rangle = 0$.
Take $\tau$ in $\bT$ such that
$\tau^{k-j}\ne 1$.
For instance,
$\tau \eqdef \enumber^{\frac{\iu\pi}{k-j}}$.
Using Lemma~\ref{lem:rho_basic}
and the assumption $S\in\cC(\rho)$, we get
\begin{align*}
\langle Sb_j,b_k\rangle
&= \langle \rho_{\tau}^\ast S \rho_{\tau} b_j, b_k \rangle
= \langle S \rho_{\tau} b_j, \rho_{\tau} b_k \rangle
\\
&= \langle S \tau^{-j} b_j, \tau^{-k} b_k\rangle
= \tau^{k-j} \langle S b_j, b_k \rangle,
\end{align*}
which holds only if $\langle S b_j,b_k\rangle = 0$.
Since $b$ is an orthonormal basis
of $A^2(\bD)$, we conclude that $S b_j = \al_j b_j$.

\medskip
($\Leftarrow$)
Let $S$ be diagonal with respect to $b$.
We denote by $\al$ the sequence of
the corresponding eigenvalues.
So, $Sb_j=\al_j b_j$ for every $j$.
For every $j$ in $\bNz$, we get
\begin{align*}
\rho_{\tau} S b_j
&= \rho_{\tau} \al_j b_j
= \al_j \rho_{\tau} b_j
= \al_j \tau^{-j} b_j,
\\[1ex]
S \rho_{\tau} b_j
&= S \tau^{-j} b_j
= \tau^{-j} S b_j
= \al_j \tau^{-j} b_j.
\end{align*}
So, $\rho_{\tau}S b_j = S\rho_{\tau} b_j$.
Since $b$ is an orthonormal
basis of $A^2(\bN)$,
we have proved that
$\rho_{\tau}S=S\rho_{\tau}$.
\end{proof}

\section{Diagonalization of the Toeplitz operator
induced by radial measure}
\label{sec:diagonalization_Toeplitz_radial_measure}

In this section,
we carry out the diagonalization of $T_\mu$
in the canonical basis,
supposing that $\mu$ is a radial Carleson--Bergman measure.

First, following ideas of Subsection~\ref{subsec:RadOper},
we check that $T_\mu$ commutes with the rotation operators.

\begin{prop}
Let $\mu$ be a radial Carleson--Bergman measure on $\bD$.
Then $T_\mu\in\cC(\rho)$.
\end{prop}

\begin{proof}
Since the measure $\mu$ is radial,
\begin{align*}
(\rho_{\tau} T_{\mu} f) (z)
&=
(T_\mu f)(\bar{\tau} z)
=
\int_{\bD}
\frac{f(w)}{(1-\bar{\tau} \bar{z}\,w)^2}\,\dif\mu(w)
\\
&=
\int_{\bD}
\frac{f(\zeta \bar{\tau})}{(1-\bar{z}\zeta)^{2}}\,\dif\mu(\bar{\tau}\zeta)
=
\int_{\bD}
\frac{f(\zeta \bar{\tau})}{(1-\bar{z}\zeta)^{2}}\,\dif\mu(\zeta)
=
(T_{\mu} \rho_{\tau} f) (z),
\end{align*}
where we used the change of variable
$\zeta=\tau w$.
\end{proof}

From Proposition~\ref{PropEquiv} we know that $T_{\mu}$ is diagonal with respect to the canonical basis;
now, the aim is to idenfity its eigenvalues,
i.e., the numbers $\al_{j}$
from Definition~\ref{DefiDiago}.

\begin{defn}
\label{def:gamma}
Given a finite Borel measure $\eta$ on $[0,1)$,
let $\ga_\eta\colon\bNz\to\bC$ be a sequence given by
\begin{equation}
\label{eq:gamma}
\ga_\eta(k)
\eqdef
2(k+1)\int_{[0,1)} r^{2k}\,\dif\eta(r).
\end{equation}
\end{defn}

\begin{thm}
\label{ThmDiagonal}
Let $\mu$ be a radial Carleson--Bergman measure on $\bD$
of the form $\mu=\nu_\eta$,
where $\eta$ is a finite Borel measure on $[0,1)$.
Then,
\[
R T_{\mu} R^\ast = M_{\ga_\eta}.
\]
\end{thm}

We prove Theorem~\ref{ThmDiagonal}
by applying $T_\mu$ to the elements of the canonical basis $b$.
In the proof of the following lemma,
we use the well-known orthogonal property of the basic Fourier functions on $\bT$:
\begin{equation}
\label{eq:orthogonal_Fourier}
\int_{\bT} \tau^{j-k}\,\dif\lambda_{\bT}(\tau)
= 2\pi \delta_{j,k},
\end{equation}
where $\delta_{j,k}$ is Kronecker's delta.

\begin{lem}
\label{LemmaTmubk}
In the assumptions of Theorem~\ref{ThmDiagonal},
for each $k$ in $\bNz$,
\begin{equation*}
T_{\mu}b_{k}=\ga_\eta(k) b_{k}.
\end{equation*}
\end{lem}

\begin{proof}[First proof]
Let $k\in\bNz$ and $z\in\bD$.
We apply the definition of $T_\mu$
and decompose of the kernel of $A^2(\bD)$
into the power series:
\begin{align*}
(T_{\mu}b_{k})(z)
&=
\frac{1}{\pi}
\int_{\bD}
\frac{b_{k}(w)}{(1-z\overline{w})^{2}}\,\dif\mu(w)
=
\frac{\sqrt{k+1}}{\pi\sqrt{\pi}}
\int_{\bD}\frac{w^{k}}{(1-z\overline{w})^{2}}\,\dif\mu(w)
\\
&=
\frac{\sqrt{k+1}}{\pi\sqrt{\pi}}
\int_{\bD}
\Biggl(\,
\sum_{j=0}^\infty
(j+1) z^j \overline{w}^j w^k
\Biggr)\,\dif\mu(w).
\end{align*}
Since $\mu$ is a finite measure and $|z|<1$,
the following series of integrals converges:
\[
\sum_{j=0}^\infty
\int_{\bD}
\bigl|(j+1)z^j \overline{w}^j w^k\bigr|\,\dif\mu(w)
\le
\mu(\bD)\,
\sum_{j=0}^\infty (j+1) |z|^j
< +\infty.
\]
Therefore, in the previous expression,
we can interchange the integral with the series:
\begin{align*}
(T_{\mu}b_{k})(z)
&=
\frac{\sqrt{k+1}}{\pi\sqrt{\pi}}
\sum_{j=0}^{\infty}(j+1)z^{j}
\int_{\bD} \overline{w}^{j}w^{k}\,\dif\mu(w).
\end{align*}
Next, we apply change of variable $w=r\tau$
and the assumption that $\mu=\nu_\eta$:
\begin{align*}
(T_{\mu}b_{k})(z)
&=
\frac{\sqrt{k+1}}{\pi\sqrt{\pi}}
\sum_{j=0}^{\infty}(j+1)z^{j}
\int_{[0,1)}\int_{\bT}
r^{j}\tau^{-j}r^{k}\tau^{k}\,
\dif\lambda_{\bT}(\tau)\,\dif\eta(r) \\
&=
\frac{\sqrt{k+1}}{\pi\sqrt{\pi}}
\sum_{j=0}^{\infty}(j+1)z^{j}
\int_{[0,1)} r^{j+k}\,\dif\eta(r)
\int_{\bT}\tau^{k-j} \dif\lambda_{\bT}(\tau).
\end{align*}
Finally, due to~\eqref{eq:orthogonal_Fourier},
\[
(T_{\mu}b_{k})(z)
=
\frac{2\sqrt{k+1}}{\sqrt{\pi}} (k+1)z^{k}
\int_{[0,1)} r^{2k}\,\dif\eta(r)
=
\gamma_\eta(k) b_{k}(z).
\qedhere
\]
\end{proof}

We now give a second proof by analyzing the inner product.

\begin{proof}
Given $j$ and $k$ in $\bNz$,
we apply~\eqref{eq:Toeplitz_measure_sesquilinear_form}
and the change of variables $z=r\tau$
to calculate
$\langle T_{\mu}b_{j}, b_{k}\rangle$:
\begin{align*}
\langle T_{\mu}b_{j}, b_{k}\rangle
&=
\int_{\bD}
b_{j}(z)\overline{b_{k}(z)}\,\dif\mu(z)
=
\frac{\sqrt{(j+1)(k+1)}}{\pi}
\int_{\bD}
z^{j}\overline{z}^k\,\dif\mu(z)
\\
&=
\frac{\sqrt{(j+1)(k+1)}}{\pi}
\int_{[0,1)}
\int_{\bT} r^{j} \tau^{j}r^{k} \tau^{-k}
\,\dif\mu_{\bT}(\tau)\,\dif\eta(r)
\\
&=
\frac{\sqrt{(j+1)(k+1)}}{\pi}
\int_{[0,1)} r^{j+k}\,\dif\eta(r)
\int_{\bT} \tau^{j-k}\,\dif\mu_{\bT}(\tau).
\end{align*}
Using~\eqref{eq:orthogonal_Fourier}
we get
$\langle T_\mu b_j,b_k\rangle = \delta_{j,k}\ga_\eta(k)$.
\end{proof}

We have now a proof of Theorem \ref{ThmDiagonal}.

\begin{proof}
It is enough to analize
$RT_{\mu}R^\ast e_{k}$ for each $k$ in $\bNz$.
However, from Proposition~\ref{PropRf} and Lemma~\ref{LemmaTmubk},
$RT_{\mu}R^\ast e_{k}= \gamma_\eta(k) e_{k}$.
\end{proof}

\begin{cor}
If $\mu_{1}$ and $\mu_{2}$
are radial Carleson--Bergman measures on $\bD$,
then the operators $T_{\mu_{1}}$ and $T_{\mu_{2}}$ commute.
\end{cor}

\section{The Berezin transform of a radial measure}
\label{sec:Berezin}

The Berezin transform
is a useful tool to analyze properties of measures.
For this reason we calculate it for radial measures, which we know from Theorem \ref{ThmERM} are of the form $\nu_\eta$ (see Definition \ref{DefRadialExt}).
So, $\eta$ represents the radial part of a radial measure.

Given a finite Borel measure $\mu$ on $\bD$,
its \emph{Berezin transform}
$\Ber_\mu\colon\bD\to[0,+\infty]$
is defined by the following formula
(see~\cite[Section~7.1]{Zhu2007}):
\begin{equation}
\label{eq:Berezin_transform_def}
\Ber_\mu(z)
\eqdef
\frac{1}{\pi}
(1-|z|^2)^2
\int_{\bD} \frac{\dif\mu(z)}{|1-\overline{z}w|^4}.
\end{equation}
In this section, we are going to compute $\Ber_{\nu_\eta}$.

\begin{lem}
\label{lem:auxiliar_integral_for_Berezin_radial}
Let $0<a<1$.
Then
\begin{equation}
\label{eq:auxiliar_integral_for_Berezin_radial}
\frac{1}{2\pi}\int_0^{2\pi}
\frac{\dif\tht}{(1-2a\cos\tht+a^2)^2}
=\frac{1+a^2}{(1-a^2)^3}.
\end{equation}
\end{lem}

\begin{proof}
Let $J$ be the left-hand side of~\eqref{eq:auxiliar_integral_for_Berezin_radial}.
For $a=0$, $J=1$,
and the equality~\eqref{eq:auxiliar_integral_for_Berezin_radial} holds.
Let $0<a<1$.
Make the change of variable
$z=\enumber^{\iu\tht}$.
Then
\[
J=\frac{1}{2\pi\iu}
\int_{\bT}
\frac{z\,\dif{}z}{(1-az)^2\,(z-a)^2}.
\]
Let $f$ be the function under the integral.
It is meromorphic on the complex plane,
holomorphic in a neighborhood of $\bT$,
and it has just one pole in $\bD$.
Namely, $a$ is a pole of $f$ of order $2$.
To compute the residue of $f$ at $a$,
we consider the following auxiliary function:
\[
h(z)
\eqdef f(z)\,(z-a)^2
=z\,(1-az)^{-2}.
\]
Its derivative is
\[
h'(z)=\frac{1+az}{(1-az)^3}.
\]
Now, we compute $J$ as the residue of $f$ at $a$:
\[
J
=\operatorname{res}(f,a)
=h'(a)
=\frac{1+a^2}{(1-a^2)^3}.
\qedhere
\]
\end{proof}

\begin{defn}
\label{DefBeta}
Given a Borel measure $\eta$ on $[0,1)$,
we define $\be_\eta\colon[0,1)\to[0,+\infty]$ by
\begin{equation}
\label{eq:beta_def}
\be_\eta(a)
\eqdef
2(1-a^2)^2
\int_{[0,1)}
\frac{(1+a^2 r^2)\,\dif\eta(r)}{(1-a^2 r^2)^3}
\qquad(a\in[0,1)).
\end{equation}
\end{defn}

\begin{prop}[Berezin transform of radial measure]
\label{prop:Berezin_transform_of_radial_measure}
Let $\eta$ be a finite Borel measure on $[0,1)$.
Then the Berezin transform of $\nu_\eta$ is given by
\[
\Ber_{\nu_\eta}(w)
=\be_\eta(|w|)\qquad(w\in\bD).
\]
\end{prop}

\begin{proof}
Let $w\in\bD$.
We represent $w$ in the form
$w=a \enumber^{\iu\tht}$,
with $a\in[0,1)$ and $\tht\in\bR$.
Then, we apply~\eqref{eq:Berezin_transform_def}
with $\mu=\nu_\eta$:
\[
\Ber_{\nu_\eta}(w)
=
\frac{1}{\pi}
(1-a^2)^2
\int_{\bD}
\frac{\dif\nu_\eta(z)}{|1-z\,a\enumber^{-\iu\phi}|^4}.
\]
Make the change of variables $z=r\enumber^{\iu\tht}$ in the integral and apply the Tonelli theorem:
\[
\Ber_{\nu_\eta}(w)
=
2(1-a^2)^2
\int_{[0,1)}
\left(
\frac{1}{2\pi}
\int_0^{2\pi}
\frac{\dif\tht}{(1-2ar\cos(\tht-\phi)+a^2 r^2)^2}
\right)\,\dif\eta(r).
\]
Since the function in the interior integral is $2\pi$-periodic, it simplifies to the integral from Lemma~\ref{lem:auxiliar_integral_for_Berezin_radial},
with $ar$ instead of $a$.
Thereby, we get $\Ber_{\nu_\eta}(w)=\be_\eta(a)$.
\end{proof}

Notice that for $\dif\eta_{0}(r)=r\,\dif{}r$, a direct computation shows that $\be_{\eta_{0}}(a)=1$ for every $a$.

In the following proposition,
we express $\be_\eta$
in terms of $\ga_\eta$,
where $\ga_\eta$ is a sequence defined
by~\eqref{eq:gamma}.
Here, we do not suppose that $\nu_\eta$ is a Carleson--Bergman measure.

\begin{prop}
\label{prop:be_via_ga}
Let $\eta$ be a finite Borel measure on $[0,1)$.
Then, for every $a$ in $[0,1)$,
\begin{equation}
\label{eq:be_via_ga}
\be_\eta(a)
=
(1-a^2)^2
\sum_{n=0}^\infty (n+1) a^{2n} \ga_\eta(n).
\end{equation}
\end{prop}

\begin{proof}
Let $R$ be the right-hand side of~\eqref{eq:be_via_ga}.
After substituting the definition~\eqref{eq:gamma} of $\ga_\eta(n)$,
we apply the well-known theorem about the series of the Lebesgue integrals of positive functions:
\begin{align*}
R
&=
(1-a^2)^2 \sum_{n=0}^\infty 2(n+1)^2 a^{2n}
\int_{[0,1)} r^{2n}\,\dif\eta(r)
\\
&= 2(1-a^2)^2
\int_{[0,1)} \left(\sum_{n=0}^\infty (n+1)^2 (a r)^{2n}\right)\,\dif\eta(r).
\end{align*}
To calculate the obtained series,
we consider the following auxiliary series
which converges on $\bD$:
\[
f(z)
\eqdef \frac{z}{(1-z)^2}
= \sum_{n=0}^\infty (n+1)z^{n+1}.
\]
Its derivative is the series that we need:
\[
\sum_{n=0}^\infty (n+1)^2 z^n
= f'(z)
= \frac{1+z}{(1-z)^3}.
\]
Therefore,
\[
R
= 2(1-a^2)^2
\int_{[0,1)} f'(ar)\,\dif\eta(r)
= \be_\eta(a).
\qedhere
\]
\end{proof}

\section{Using averages and distribution functions}\label{sec:AF}

It is known that the most important properties of a Toeplitz operator $T_a$ in $A^2(\bD)$
are determined by the behavior of $a$ near the boundary.
Grudsky and Vasilevski~\cite{GrudskyVasilevski2001}
in their study of radial Toeplitz operators
introduced integrals of an integrable function near the boundary.
We extended their idea to radial measures.
In our case, an important tool that we use is the average values, which will be introduced in Definition \ref{DefAver}.
We take $\eta$ to be the radial part of a radial measure, see Theorem \ref{ThmERM}

In this section we consider the distribution function associated to a measure. 
This helps to work with the integration by parts formula associated to measures. 
A good introduction to these ideas can be found
in Fleming and Harrington~\cite[Appendix~A]{FlemingHarrington2005}.

\begin{defn}
Given a finite Borel measure $\eta$ on $[0,1)$, we consider its extension on $(-\infty, \infty)$ such that $\eta([0,1)^{c})=0$.
Now, define the associated distribution function as
\[
F_\eta(u)
\eqdef \eta((-\infty,u]).
\]
\end{defn}

It is known that $F_{\eta}:(-\infty, \infty)\to [0,\infty)$ is non-decreasing and right-continuous function.
We also define its left-continous version as
\[
F_{\eta}^{-}(u)
\eqdef
\lim_{\substack{r\to u\\r<u}}F_\eta(r)
= \eta((-\infty,u)).
\]
Notice that
\[
F_\eta(0) = \eta(\{0\}),\qquad
F_{\eta}^{-}(1) = F_{\eta}(1)= \eta([0,1)).
\]
Now we introduce the \emph{average} of the the radial measure $\eta$ on the interval $[r,1)$.
The term \emph{average} comes from using the following fact.
If $\lambda$ represents the Lebesgue measure on $\bD$, then its radial part is $d\eta_{0}(r)= r\dif r$, thus
$\eta_{0}([r,1))= \frac{1-r^{2}}{2}$.

\begin{defn}\label{DefAver}
Given a finite Borel measure $\eta$ on $[0,1)$,
let $\ka_\eta\colon [0,1)\to[0,+\infty)$ be defined by
\[
\ka_\eta(r)
\eqdef 
\frac{\eta([r,1))}{\eta_{0}([r,1))}=
\frac{2\eta([r,1))}{1-r^2}.
\]
\end{defn}

Equivalently,
\begin{equation}
\label{eq:average_via_distribution}
\ka_\eta(r)
=\frac{2}{1-r^2} (F_{\eta}(1)-F_{\eta}^{-}(r)).
\end{equation}


\begin{rem}[integration by parts for a Riemann--Stieltjes integral with a continuously differentiable function]
Let $f$ be a Borel real-valued function on $(-\infty,\infty)$.
The Stieltjes and Lebesgue integrals are related as follows:
\[
\int_{a}^{b}f(r)\,\dif F_{\eta}(r)
=\int_{(a,b]}f\,\dif \eta,
\]
for $a<b$.
If $f$ is continuously differentiable,
the integration by parts formula tells us that (see Theorem A.1.2 in \cite{FlemingHarrington2005})
\begin{align*}
f(b)F(b)-f(a)F(a)
&=
\int_{a}^{b}f(r)\,\dif F(r)
+\int_{a}^{b}F_{\eta}(r)\,\dif f(r)
\\
&=
\int_{a}^{b}f(r)\,\dif F(r)
+\int_{a}^{b}F_{\eta}^{-}(r)\,\dif f(r).
\end{align*}
\end{rem}

Using previous remark we can then obtain the 
\begin{prop}[integration by parts in our settings]
\label{prop:our_integration_by_parts}
For $u\in [0,1)$,
\begin{align*}
\int_{[0,u)} f\,\dif\eta
&= f(u) \eta([0,u)) - \int_0^u f'(r) F_\eta(r)\,\dif{}r
\\
&= f(u) \eta([0,u)) - \int_0^u f'(r) F_\eta^{-}(r)\,\dif{}r
\end{align*}
\end{prop}

\begin{proof}
Using the integration by parts formula in previous remark, with $a=0$ and $b=u$,
\begin{align*}
\int_{[0,u)}f\dif \eta
&=
\int_{ \{0\}}f\dif \eta+ \int_{ (0,u] }f\dif \eta- \int_{ \{u\}}f\,\dif \eta
\\[0.5ex]
&=
f(0)\eta(\{0\}) - f(u)\eta(\{u\})
\\
&\qquad + f(u)F_{\eta}(u)- f(0)F_{\eta}(0)- \int_{0}^{u}F_{\eta}(r)\,\dif f(r)
\\
&=
f(u) F_{\eta}^{-}(u) - \int_{0}^{u}F_{\eta}(r) f^{\prime}(r)\, \dif r
\\
&=
f(u) F_{\eta}^{-}(u) - \int_{0}^{u}F_{\eta}^{-}(r) f^{\prime}(r)\,\dif r.
\end{align*}
which give rise to both equalities.
\end{proof}

We are in position to give the relation with the average.
\begin{prop}[integration by parts via the averages]
\label{prop:integration_by_parts_via_averages}
It holds that
\begin{equation}
\label{eq:integral_by_parts_via_averages}
\int_{[0,1)} f\,\dif\eta
= f(0)F_{\eta}(1)+ \int_{0}^{1}\frac{1-r^{2}}{2}f^{\prime}(r)\ka_{\eta}(r)\dif r.
\end{equation}
\end{prop}

\begin{proof}
From the second equality of Proposition (\ref{prop:our_integration_by_parts}) with $u=1$, we have
\begin{equation}\label{EqSIPA}
\int_{[0,1)} f\,\dif\eta
= f(1)F_{\eta}^{-}(1)- \int_{0}^{1}F_{\eta}^{-}(r)f^{\prime}(r)\dif r.
\end{equation}
And from (\ref{eq:average_via_distribution}) we have 
\[
F_{\eta}^{-}(r)= \eta([0,1))- \frac{1-r^{2}}{2}\ka_{\eta}(r).
\]
Then plugging previous display into \eqref{EqSIPA} yields
\begin{align*}
\int_{[0,1)} f\,\dif\eta 
&=  f(1)F_{\eta}(1)-\int_{0}^{1}\left\{ \eta([0,1))
- \frac{1-r^{2}}{2}\ka_{\eta}(r)\right\}f^{\prime}(r)\,\dif{}r
\\
&= f(1)F_{\eta}(1)-F(1) \int_{0}^{1}f^{\prime}(r)\,\dif{}r
+ \int_{0}^{1} \frac{1-r^{2}}{2}\ka_{\eta}(r)f^{\prime}(r)\,\dif r,
\end{align*}
which gives \eqref{eq:integral_by_parts_via_averages}.
\end{proof}

We apply previous ideas to the spectral sequence $\ga_\eta$.

\begin{prop}[sequence $\ga_\eta$ in terms of $F_\eta$ and $\ka_{\eta}$]
\label{prop:gamma_via_distribution_and_averages}
Let $\eta$ be a finite Borel measure on $[0,1)$.
For every $n$ in $\bN$,
\begin{equation}
\label{eq:gamma_via_distribution}
\ga_\eta(n)
= 2(n+1)\eta([0,1))
- 4n(n+1) \int_0^1 F_\eta(r) r^{2n-1}\,\dif{}r,
\end{equation}
and
\begin{equation}
\label{eq:gamma_via_averages}
\ga_\eta(n)
=2n(n+1) \int_0^1 \ka_\eta(r) r^{2n-1}(1-r^2)\,\dif{}r.
\end{equation}
Moreover,
\begin{equation}
\label{eq:gamma_0_via_everage}
\ga_\eta(0)
= 2\eta([0,1))
= 2F_\eta(1)
= \ka_\eta(0).
\end{equation}
\end{prop}

\begin{proof}
It follows from Propositions~\ref{prop:our_integration_by_parts}
and~\ref{prop:integration_by_parts_via_averages}
applied with $f(r)=r^{2n}$.
\end{proof}

Finally, we present the conexion to the Berezin transform, through the function $\be_{\eta}$ of Definition \ref{DefBeta}.
\begin{prop}[$\be_\eta$ in terms of $\ka_\eta$]
For every $a$ in $[0,1)$,
\begin{equation}
\label{eq:beta_via_averages}
\be_\eta(a)
=
2(1-a^{2})^{2}F(1)+
4a^2(1-a^2)^2
\int_0^1
\ka_\eta(r)\,
\frac{(2+a^2 r^2)(1-r^2)r}{(1-a^2 r^2)^4}\,\dif{}r.
\end{equation}
\end{prop}

\begin{proof}
Let $a\in[0,1)$.
We start with equation \eqref{eq:beta_def} and consider $f\colon[0,1)\to\bR$ to be
\[
f(r)
\eqdef
\frac{1+a^2 r^2}{(1-a^2 r^2)^3}.
\]
Its derivative is
\[
f'(r)
=\frac{4a^2 (2+a^2 r^2)r}{(1-a^2 r^2)^4}.
\]
By Proposition~\ref{prop:integration_by_parts_via_averages},
we obtain~\eqref{eq:beta_via_averages}.
\end{proof}

\begin{ex}[the radial measure corresponding to the identity Toeplitz operator]
For $\dif\eta_{0}(r)=r\,\dif{}r$,
\begin{align*}
\ga_{\eta_0}(n)
&=
2n(n+1) \int_0^1 r^{2n-2}(1-r^2)\,\dif{}\eta_{0}(r)
\\
&=
2n(n+1) \int_0^1 r^{2n-1}(1-r^2)\,\dif{}r = (n+1) - n
=1.
\end{align*}
\end{ex}

\begin{ex}[Dirac measure]
Let $x\in(0,1)$
and $\delta_x$ be the Dirac measure concentrated at $x$.
This means that
$\delta_x(\{x\})=1$
and $\delta_x([0,1)\setminus\{x\})=0$.
We also consider $\delta_x$ as a measure on the whole real line.
For $\eta=\delta_x$,
the corresponding distribution function is the characteristic function of $[x,+\infty)$, and
\[
\ka_{\delta_x}(r)
=
\begin{cases}
\frac{2}{1-r^2}, & 0\le r\le x;
\\
0, & x<r<1.
\end{cases}
\]
It is easy to check that formulas~\eqref{eq:gamma},
\eqref{eq:gamma_via_distribution}
and \eqref{eq:gamma_via_averages}
yield the same expression for $\ga_{\delta_x}$:
\[
\ga_{\delta_x}(n) = 2(n+1) x^{2n}\qquad(n\in\bNz).
\]
Furthermore, it is possible to verify that
\eqref{eq:beta_def}
and~\eqref{eq:beta_via_averages}
yield the same expression for $\be_{\delta_x}$:
\[
\be_{\delta_x}(a)
=
2(1-a^2)^2
\frac{(1+a^2 x^2)}{(1-a^2 x^2)^3}.
\]
\end{ex}

\section{When a radial measure is Carleson--Bergman measure?}
\label{sec:radial_measure_criterion_Carleson}

Given a finite Borel measure $\mu$ on $\bD$,
it is known (see Zhu~\cite[Theorem~7.5]{Zhu2007})
that $\mu$ is a Carleson--Bergman measure,
if and only if its Berezin transform
$\Ber_\mu$ is a bounded function.
Moreover, this condition is necessary and sufficient for $T_\mu$ to be bounded.

As we show in the following theorem,
this condition can be essentially simplified for radial measures.
In particular, we get a very simple criterion in terms of the function $\ka_\eta$.
The use of the averaging function is inspired by~\cite{GrudskyVasilevski2001},
but we use a different normalization
and work with measures (instead of integrable functions).

\begin{thm}
\label{thm:criterion_radial_Carleson_Bergman}
Let $\eta$ be a finite Borel measure on $[0,1)$
and $\mu=\nu_\eta$.
Then, the following conditions are equivalent:
\begin{itemize}
\item[(a)] $T_\eta$ is a bounded operator;
\item[(b)] $\be_\eta$ is a bounded function;
\item[(c)] $\ga_\eta$ is a bounded sequence;
\item[(d)] $\ka_\eta$ is a bounded function.
\end{itemize}
\end{thm}

The rest of this section is a proof of this theorem, divided in three lemmas.

\begin{lem}
\label{lem:trivial_comparison_be_ga_ka}
Let $\eta$ be a finite Borel measure on $[0,1)$.
Then,
\begin{equation}
\label{eq:trivial_comparison_be_ga_ka}
\|\be_\eta\|_\infty
\le \|\ga_\eta\|_\infty
\le \|\ka_\eta\|_\infty.
\end{equation}
\end{lem}

\begin{proof}
1. For every $a$ in $[0,1)$,
by Proposition~\ref{prop:be_via_ga},
\[
\be_\eta(a)
\le
\|\ga_\eta\|_\infty\cdot
(1-a^2)^2
\sum_{n=0}^\infty (n+1) a^{2n}
=
\|\ga_\eta\|_\infty.
\]
Hence,
$\|\be_\eta\|_\infty\le \|\ga_\eta\|_\infty$.

\medskip\noindent
2. For every $n$ in $\bN$,
by~\eqref{eq:gamma_via_averages},
\[
\ga_\eta(n)
\le \|\ka_\eta\|_\infty
=2n(n+1) \int_0^1 r^{2n-1}(1-r^2)\,\dif{}r
=\|\ka_\eta\|_\infty.
\]
For $n=0$, by~\eqref{eq:gamma_0_via_everage},
\[
\ga_\eta(0)
= \ka_\eta(0)
\le \|\ka_\eta\|_\infty.
\]
Thus, $\|\ga_\eta\|_\infty \le \|\ka_\eta\|_\infty$.
\end{proof}

The next technical lemma provides a rough approximation
of the maximum value of the sequence
$n \mapsto (1-s^2)(n+1)s^{2n}$
which appears
in the proof of Lemma~\ref{lem:bound_ka_via_ga}.

\begin{lem}
\label{lem:m_s_inequality}
Let $\frac{3}{4}\le s\le 1$
and
\[
m \eqdef
\left\lfloor
\frac{1}{2(1-s)}
\right\rfloor.
\]
Then
\begin{equation}
\label{eq:m_s_expression}
(1-s^2) (m + 1) s^{2m} > \frac{1}{4}.
\end{equation}
\end{lem}

\begin{proof}
By the definition of the floor function
(the entire part of real numbers),
\[
m \le \frac{1}{2(1-s)}
\qquad\text{and}\qquad
m + 1 > \frac{1}{2(1-s)}.
\]
Let $X$ be the expression in the left-hand side of~\eqref{eq:m_s_expression}.
Then
\begin{align*}
X
&\ge
(1-s^2)\,\frac{1}{2(1-s)} s^{\frac{1}{1-s}}
\ge
\frac{s+1}{2} s^{\frac{1}{1-s}}.
\end{align*}
An elementary analysis shows that
$s \mapsto \ln(s)/(1-s)$
is a strictly increasing function on $(0,1)$.
Therefore,
\[
\frac{s+1}{2}\,s^{\frac{1}{1-s}}
\ge \frac{7}{8}\cdot
\left(\frac{3}{4}\right)^4
=\frac{567}{2048}
>\frac{1}{4}.
\qedhere
\]
\end{proof}

\begin{lem}
\label{lem:bound_ka_via_ga}
Let $\eta$ be a finite Borel measure on $[0,1)$.
Then,
\begin{equation}
\label{eq:bound_ka_via_ga}
\|\ka_\eta\|_\infty
\le 5 \|\ga_\eta\|_\infty.
\end{equation}
\end{lem}

\begin{proof}
Let $s$ be a fixed number such that $s\in[0,1)$.
We are going to prove that
$\ka_\eta(s)\le 5\|\ga_\eta\|_\infty$.
For $0\le s\le \frac{3}{4}$,
the situation is simple.
By~\eqref{eq:gamma_0_via_everage},
\[
\ka_\eta(s)
= \frac{2\eta([s,1))}{1-s^2}
\le \frac{16}{7}\cdot 2\eta([0,1))
\le 5\ga_\eta(0)
\le 5\|\ga_\eta\|_\infty.
\]
Consider the main case:
$\frac{3}{4}\le s<1$.
For every $r$ with $s\le r<1$,
we get the following lower estimate
of $\ka_\eta(r)$
in terms of $\ka_\eta(s)$:
\begin{equation}
\label{eq:ka_eta_r_ka_eta_s}
\ka_\eta(r)
=\frac{2\eta([r,1))}{1-r^2}
\ge \frac{2\eta([s,1))}{1-r^2}
=
\frac{1-s^2}{1-r^2}
\cdot
\frac{2\eta([s,1))}{1-s^2}
= \frac{1-s^2}{1-r^2}\,\ka_\eta(s).
\end{equation}
For every $n$ in $\bN$,
we apply~\eqref{eq:gamma_via_averages}
and~\eqref{eq:ka_eta_r_ka_eta_s}.
Thereby, we estimate $\ga_\eta(n)$
from below in terms of $\ka_\eta$:
\begin{align*}
\ga_\eta(n)
&\ge
2n(n+1)\int_{[0,s)} \ka_\eta(r) r^{2n-1}(1-r^2)\,\dif{}r
\\
&\ge 2n(n+1)
\int_{[0,s)}\ka_\eta(s) (1-s^2)r^{2n-1}\,\dif{}r
= (1-s^2) (n+1) s^{2n}\,\ka_\eta(s).
\end{align*}
Recall that $s$ is a fixed number.
We choose $m$ as in Lemma~\ref{lem:m_s_inequality} and get
\[
\ga_\eta(m)
\ge \frac{1}{4} \ka_\eta(s).
\]
Thus, $\ka_\eta(s)\le 4\ga_\eta(m) \le 4\,\|\ga_\eta\|_\infty$.
Joining the cases $s<\frac{3}{4}$ and $s\ge\frac{3}{4}$,
we obtain
$\ka_\eta(s)\le 5\|\ga_\eta\|_\infty$.
\end{proof}

\begin{proof}[Proof of Theorem~\ref{thm:criterion_radial_Carleson_Bergman}]
By Proposition~\ref{prop:Berezin_transform_of_radial_measure},
$\|\Ber_{\nu_\eta}\|_\infty = \|\be_\eta\|_\infty$.
So, by the general criterion,
mentioned in the beginning of this section,
conditions (a) and (b) are equivalent.
Lemmas~\ref{lem:trivial_comparison_be_ga_ka}
and~\ref{lem:bound_ka_via_ga}
show the equivalence of (b), (c) and (d).
\end{proof}

\section{Lipschitz property of the spectral sequences}
\label{sec:Lip}

Now, we study the behavior of the spectral sequence $\ga_\eta$.
Following~\cite{GrudskyMaximenkoVasilevski2013}, we consider the following ``logarithmic distance'' on $\bNz$:
\[
d_{\log}(m,n)
\eqdef
\bigl|\log(m+1)-\log(n+1)\bigr|.
\]
It was proved in~\cite{Suarez2008}
and~\cite{GrudskyMaximenkoVasilevski2013}
that if $a\in L_\infty([0,1))$,
then $\ga_a$ is Lipschitz continuous with respect to $d_{\log}$.
In this section, we prove that this property extends to $\ga_\eta$, supposing that $\eta$ is a (positive) Borel measure on $[0,1)$ such that $\ka_\eta$ is bounded.

We denote by $L$ the ``integral kernel'' in the formula~\eqref{eq:gamma_via_averages}:
\[
L(n,r) \eqdef 2n(n+1) r^{2n-1} (1-r^2).
\]

\begin{lem}
\label{lem:auxiliary_integral_for_lip}
For every $n$ in $\bN$,
\begin{equation}
\label{eq:integral_abs_dif_L_minus_L_next}
\int_0^1
\bigl|L(n+1,r)-L(n,r)\bigr|\,\dif{}r
=\frac{8(n+1)n^n}{(n+2)^{n+2}}.
\end{equation}
\end{lem}

\begin{proof}
Let $n\in\bN$.
We denote by $J_n$ the integral in the left-hand side of~\eqref{eq:integral_abs_dif_L_minus_L_next}.
First, we notice that
the expression $L(n+1,r)-L(n,r)$
is negative on $(0,r_0)$ and positive on $(r_0,1)$,
where
\[
r_0 \eqdef \sqrt{\frac{n}{n+2}}.
\]
Therefore, $J_n$
simplifies in the following manner:
\[
J_n
=
\int_0^{r_0} (L(n,r)- L(n+1,r))\,\dif{}r
+ \int_{r_0}^1 (L(n+1,r) - L(n,r))\,\dif{}r.
\]
Consider the following antiderivative of $L(n,\cdot)$:
\begin{align*}
\widetilde{L}(n,x)
&\eqdef \int_0^x L(n,r)\,\dif{}r
=2n(n+1) \int_0^x (r^{2n-1} - r^{2n+1})\,\dif{}r
\\[0.5ex]
&= (n+1)x^{2n} - nx^{2n+2}.
\end{align*}
Since $\widetilde{L}(n,0)=\widetilde{L}(n+1,0)=0$
and $\widetilde{L}(n,1)=\widetilde{L}(n+1,1)=1$,
\begin{align*}
J_n
&= 2(\widetilde{L}(n,r_0)-\widetilde{L}(n+1,r_0))
\\[0.5ex]
&=2(n+1)2(n+1)(1-2r_0^2+r_0^4)r_0^{2n}
\\[0.5ex]
&=2(n+1)\left(1-\frac{n}{n+2}\right)^2\,\left(\frac{n}{n+2}\right)^n.
\end{align*}
After a simplification,
we obtain the right-hand side of~\eqref{eq:integral_abs_dif_L_minus_L_next}.
\end{proof}

The following lemma is auxiliary.

\begin{lem}
For each $n$ in $\bN$,
\[
\frac{1}{m+1} \le \log(m+1) - \log(m).
\]
\end{lem}

\begin{proof}
For each $x$ in $[m,m+1]$, we have
$\frac{1}{m+1}\le \frac{1}{x}$.
Therefore,
\[
\frac{1}{m+1}
\le \int_m^{m+1} \frac{\dif{}x}{x}
= \log(m+1) - \log(m).
\qedhere
\]
\end{proof}

The have now the main result of this section.

\begin{thm}
\label{thm:Lip}
Let $\eta$ be a Borel measure on $[0,1)$
such that $\ka_\eta$ is a bounded function.
Then, $\ga_\eta$ is Lipschitz continuous
with respect to $d_{\log}$.
\end{thm}

\begin{proof}
First, we are going to bound from above the difference $|\ga_\eta(n+1)-\ga_\eta(n)|$.
For each $n$ in $\bN$,
\begin{align*}
|\ga_\eta(n+1)-\ga_\eta(n)|
&\le \int_0^1 |L(n+1,r)-L(n,r)|\,
\ka_\eta(r)\,\dif{}r
\\[0.5ex]
&\le \frac{8(n+1)n^n}{(n+2)^{n+2}} \|\ka_\eta\|_\infty
\le \frac{8\|\ka_\eta\|_\infty}{n+2}
\\[0.5ex]
&\le 8\|\ka_\eta\|_\infty
\bigl(\log(n+2) - \log(n+1)\bigr).
\end{align*}
For $n=0$, we obtain a similar estimate in a more simple way:
\begin{align*}
|\ga_\eta(1)-\ga_\eta(0)|
&
\le \ga_\eta(0) + \ga_\eta(1)
=\ka_\eta(0)
+ 4 \int_0^1 \ka_\eta(r)\,r(1-r^2)\,\dif{}r
\\
&
\le 2\|\ka_\eta\|_\infty
\le 8\|\ka_\eta\|_\infty (\log(2)-\log(1)).
\end{align*}
So, for every $m,n$ in $\bNz$ with $m<n$, we have
\begin{align*}
|\ga_\eta(m)-\ga_\eta(n)|
&\le
\sum_{k=n}^{m-1} |\ga_\eta(k+1)-\ga_\eta(k)|
\\[0.5ex]
&\le
8\|\ga_\eta\|_\infty (\log(m+1)-\log(n+1))
\\[0.5ex]
&=
8\|\ga_\eta\|_\infty d_{\log}(m,n).
\qedhere
\end{align*}
\end{proof}

\section{A class of radial complex measures}
\label{sec:class_of_radial_complex_measures}

In this section,
we propose a trivial extension
of some of the previous results
to radial complex measures.

\begin{defn}
\label{def:class_of_complex_measures}
We denote by $\mathfrak{M}_{\bC}$
the set of all finite complex Borel measures
$\eta$ on $[0,1)$ of the form
\begin{equation}
\label{eq:eta_complex}
\eta = \eta_1 - \eta_2 + \iu\,(\eta_3 - \eta_4),
\end{equation}
where $\eta_j$ are (positive) finite Borel measures on $[0,1)$ such that the functions $\ka_{\eta_j}$ are bounded for each $j$ in $\{1,2,3,4\}$.
\end{defn}

Given $\eta$ in $\mathfrak{M}_{\bC}$,
we define
$\nu_\eta$, $T_{\nu_\eta}$,
$\ga_\eta$ and $\be_\eta$
by the same formulas are before.

\begin{cor}
\label{cor:complex_case}
Let $\eta\in\mathfrak{M}_{\bC}$.
Then the following properties hold.
\begin{enumerate}
\item $\nu_\eta$ is a radial measure.
\item $T_{\nu_\eta}$ is a bounded radial operator,
and $R T_{\nu_\eta} R^\ast = M_{\ga_\eta}$.
\item $\ga_\eta$ is a bounded sequence.
\item $\ka_\eta$ and $\be_\eta$ are bounded functions.
\item $\ga_\eta$ is Lipschitz continuous with respect to $\rho$.
\end{enumerate}
\end{cor}

\begin{proof}
This is a corollary of Theorems~\ref{ThmERM},
\ref{ThmDiagonal},
\ref{thm:criterion_radial_Carleson_Bergman},
and \ref{thm:Lip}.
Indeed, if $\eta$ is of the form~\eqref{eq:eta_complex}
from Definition~\ref{def:class_of_complex_measures},
then it is easy to see that
\[
\nu_\eta = \nu_{\eta_1}-\nu_{\eta_2}+\iu\,(\nu_{\eta_3}-\nu_{\eta_3}),
\quad
T_{\nu_\eta} = T_{\nu_{\eta_1}}-T_{\nu_{\eta_2}}
+\iu\,(T_{\nu_{\eta_3}} - T_{\nu_{\eta_3}}),
\]
and similar decompositions also holds for $\ga_\eta$, $\be_\eta$, and $\ka_\eta$.
\end{proof}

Let us denote by
$\operatorname{BLip}(\bNz,d_{\log})$
the set of all bounded complex sequences that are Lipschitz continuous with respect to $d_{\log}$,
and by
$\operatorname{BUC}(\bNz,d_{\log})$ the set of all bounded complex sequences that are uniformly continuous with respect to $d_{\log}$.
These sequence classes (in another notation)
were introduced in~\cite{GrudskyMaximenkoVasilevski2013}.

\begin{cor}
\label{cor:gammas_induced_by_a_class_of_complex_radial_measures}
The set of sequences
$\{\ga_\eta\colon\ \eta\in\mathfrak{M}_\bC\}$
is a dense subset of
$\operatorname{BUC}(\bNz,d_{\log})$.
The C*-subalgebra of $l_\infty(\bNz)$
generated
by $\{\ga_\eta\colon\ \eta\in\mathfrak{M}_\bC\}$
coincides with
$\operatorname{BUC}(\bNz,d_{\log})$.
\end{cor}

\begin{proof}
In the context of general metric spaces,
it is obvious and well known
that every Lipschitz continous function
is uniformly continuous.
Therefore,
$\operatorname{BLip}(\bNz,d_{\log})
\subseteq
\operatorname{BUC}(\bNz,d_{\log})$.

It was proved in~\cite{GrudskyMaximenkoVasilevski2013}
(see also~\cite{Suarez2008,
BauerHerreraYanezVasilevski2014,
HerreraMaximenkoVasilevski2015})
that $\{\ga_a\colon\ a\in L_\infty([0,1))\}$
is a dense subset of
$\operatorname{BUC}(\bNz,d_{\log})$.

For every $a$ in $L^\infty([0,1))$,
the measure $\mu$ defined by
$\dif\mu(z)=a(z)\,\dif\lambda(z)$
belongs to the class $\mathfrak{M}_\bC$;
this follows from the decomposition of $a$
into a linear combination of four positive functions.
Therefore, the following inclusions hold:
\[
\{\ga_a\colon\ a\in L_\infty([0,1))\}
\subseteq
\{\ga_\eta\colon\ \eta\in\mathfrak{M}_\bC\}
\subseteq
\operatorname{BLip}(\bNz,d_{\log})
\subseteq
\operatorname{BUC}(\bNz,d_{\log}).
\]
Since
$\{\ga_a\colon\ a\in L_\infty([0,1))\}$
is a dense subset of
$\operatorname{BUC}(\bNz,d_{\log})$,
we conclude that
$\{\ga_\eta\colon\ \eta\in\mathfrak{M}_\bC\}$
also is a dense subset of
$\operatorname{BUC}(\bNz,d_{\log})$.
Moreover,
since $\operatorname{BUC}(\bNz,d_{\log})$
is a C*-subalgebra of $l_\infty(\bNz)$,
it is exactly the C*-subalgebra
generated by
$\operatorname{BUC}(\bNz,d_{\log})$.
\end{proof}

Probably, it is possible to define bounded radial Toeplitz operators $T_{\nu_\eta}$
for complex measures $\eta$
from a wider class than $\mathfrak{M}_{\bC}$.
It is an interesting theme for a future research.


\subsection*{Funding}

The research of the first author has been supported by
SECIHTI, Mexico,
project ``Ciencia de Frontera'' FORDECYT-PRONACES/61517/2020,
and by Instituto Polit\'{e}cnico Nacional,
Mexico (Proyectos IPN-SIP).

\subsection*{Acknowledgement}

The first author is grateful to Ondrej Hutn\'{i}k and Crispin Herrera-Ya\~{n}ez for discussing ideas about ``vertical'' Bergman--Carleson measures (that project is still unfinished).

\end{document}